\documentclass[leqno,12pt]{article} 
\usepackage{amsmath, amssymb}
\usepackage{amsthm} 
\theoremstyle{plain} 
\newtheorem{theorem}{\indent\sc Theorem}[section]

\newtheorem{corollary}[theorem]{\indent\sc Corollary}
\newtheorem{proposition}[theorem]{\indent\sc Proposition}

\theoremstyle{definition} 
\newtheorem{definition}[theorem]{\indent\sc Definition}
\newtheorem{remark}[theorem]{\indent\sc Remark}
\newtheorem{example}[theorem]{\indent\sc Example}

\title{$\eta$-Ricci solitons in $(\varepsilon)$-almost paracontact metric manifolds} 

\author{Adara Monica Blaga, Selcen Y\"{u}ksel Perkta\c{s}, \\
Bilal Eftal Acet and
Feyza Esra Erdo\u{g}an}

\date{}
\begin{document}

\maketitle

\begin{abstract}
The object of this paper is to study $\eta $-Ricci
solitons on $\left( \varepsilon \right) $-almost paracontact metric
manifolds. We investigate $\eta $-Ricci solitons in the case when its
potential vector field is exactly the characteristic vector field $\xi $ of
the $\left( \varepsilon \right) $-almost paracontact metric manifold and
when the potential vector field is torse-forming. We also study Einstein-like and $\left(
\varepsilon \right) $-para Sasakian manifolds admitting $\eta $-Ricci
solitons. Finally we obtain some results for $\eta $-Ricci solitons on $%
\left( \varepsilon \right) $-almost paracontact metric manifolds with a special view towards
parallel symmetric $\left( 0,2\right) $-tensor fields.
\end{abstract}

\medskip

\noindent {\bf Mathematics Subject Classification:} \medskip 53C15, 53C25,
53C40, 53C42, 53C50.

\medskip


\medskip

\section{Introduction}

The notion of Ricci soliton which is a natural generalization of an Einstein
metric (i.e. the Ricci tensor $S$ is a constant multiple of $g$) was
introduced by Hamilton \cite{Hamilton-1982} in 1982. A pseudo-Riemannian
manifold $(M,g)$ is called a \textit{Ricci soliton} if it admits a smooth vector
field $V$ (potential vector field) on $M$ such that
\begin{equation}
\frac{1}{2}\left( \pounds _{V}\,g\right) \left( X,Y\right) +S(X,Y)+\lambda
g(X,Y)=0,  \label{int-1}
\end{equation}%
where $\pounds _{V}$ denotes the Lie-derivative in the direction $V,$ $%
\lambda $ is a constant and $X$, $Y$ are arbitrary vector fields on $M$. A
Ricci soliton is said to be shrinking, steady or expanding according to $%
\lambda $ being negative, zero or positive respectively. It is obvious that
a trivial Ricci soliton is an Einstein manifold with $V$ zero or Killing vector field.
Since Ricci solitons are the fixed points of the Ricci flow, they are
important in understanding Hamilton's Ricci flow \cite{Hamilton-1988}: $%
\frac{\partial }{\partial t}g_{ij}=-2S_{ij}$, viewed as a dynamical system,
on the space of Riemannian metrics modulo diffeomorphisms and scalings. In
differential geometry, the Ricci flow is an intrinsic geometric flow. It can be viewed as
a process that deforms the metric of a Riemannian manifold in a way formally
analogous to the diffusion of heat, smoothing out the irregularities in the
metric.

Geometric flows, especially Ricci flows, have become important tools in
theoretical physics. Ricci soliton is known as quasi Einstein metric in
physics literature \cite{Friedan-1985} and the solutions of the Einstein field
equations correspond to Ricci solitons \cite{Akbar-Woolgar-2009}. Relation
with the string theory and \ the fact that (\ref{int-1}) is a particular
case of Einstein field equation makes the equation of Ricci soliton
interesting in theoretical physics.

In spite of introducing and studying firstly in Riemannian geometry, the
Ricci soliton equation has recently been investigated in pseudo-Riemannian
context, especially in Lorentzian case \cite{Blaga3, Brozos-2012, Case-2010, Mat-89}.

The concept of $\eta $-Ricci soliton was initiated by Cho and Kimura \cite%
{Cho-Kimura}. An \textit{$\eta $-Ricci soliton} is a data $\left( g,V,\lambda ,\mu
\right) $ on a pseudo-Riemannian manifold satisfying
\begin{equation}
\frac{1}{2}\left( \pounds _{V}\,g\right) \left( X,Y\right) +S(X,Y)+\lambda
g(X,Y)+\mu \eta \otimes \eta \left( X,Y\right) =0,  \label{int-2}
\end{equation}%
where $\pounds _{V}$ denotes the Lie-derivative in the direction $V,$ $S$
stands for the Ricci tensor field, $\lambda $ and $\mu $ are constants and $%
X $, $Y$ are arbitrary vector fields on $M$. In \cite{Calin-Crasma} the
authors studied $\eta $-Ricci solitons on Hopf hypersurfaces in complex
space forms. In the context of paracontact geometry $\eta $-Ricci solitons
were investigated in \cite{Blaga1, Blaga2, Blaga3}.

In 1923, Eisenhart \cite{Eisenhart} proved that if a Riemannian manifold admits a second order parallel symmetric covariant tensor which is not a
constant multiple of the metric tensor, then the manifold is reducible. In
1925, it was shown by Levy \cite{Levy} that a second order parallel
symmetric non-degenerate tensor field in a space form is proportional to the
metric tensor. Also if a (pseudo-)Riemannian manifold admits a parallel
symmetric (0,2)-tensor field, then it is locally the direct product of a
number of (pseudo-)Riemannian manifold \cite{Einshart36}. Sharma \cite%
{Einshart26} studied second order parallel tensors by using Ricci
identities. Second order parallel tensors have been studied by various
authors in different structures of manifolds \cite{Einshart6, Eisenhart, Levy, Einshart19, Einshart26, Einshart27, Einshart28}.
\pagebreak

In $1976$, S\={a}to \cite{Sato-76} introduced the almost paracontact structure as a triple $(\varphi ,\xi ,\eta )$ of a (1,1)-tensor field $\varphi$, a vector field $\xi$ and a $1$-form $\eta$ satisfying $%
\varphi ^{2}=I-\eta \otimes \xi $ and $\eta (\xi )=1$. The structure is an
analogue of the almost contact structure \cite{Sasaki-60-Tohoku} and is
closely related to almost product structure (in contrast to almost contact
structure, which is related to almost complex structure). An almost contact
manifold is always odd-dimensional but an almost paracontact manifold could
be even-dimensional as well. In $1969$, Takahashi \cite%
{Takahashi-69-Tohoku-1} introduced almost contact manifolds equipped with an
associated pseudo-Riemannian metric and, in particular, he studied Sasakian
manifolds equipped with an associated pseudo-Riemannian metric. These
indefinite almost contact metric manifolds and indefinite Sasakian manifolds
are also known as $\left( \varepsilon \right) $-almost contact metric
manifolds and $\left( \varepsilon \right) $-Sasakian manifolds, respectively
\cite{Bej-Dug-93}. In 1989, Matsumoto \cite{Mat-89} replaced the
structure vector field $\xi $ by $-\xi $ in an almost paracontact manifold
and associated a Lorentzian metric with the resulting structure and called
it a Lorentzian almost paracontact manifold.

An $(\varepsilon )$-Sasakian manifold is always odd-dimensional. On the
other hand, in a Lorentzian almost paracontact manifold given by
Matsumoto, the pseudo-Riemannian metric has only index $1$ and the
structure vector field $\xi $ is always timelike. These circumstances
motivated the authors of \cite{Tri-KYK-10} to associate a pseudo-Riemannian
metric, not necessarily Lorentzian, with an almost paracontact structure,
and this indefinite almost paracontact metric structure was called an $\left(
\varepsilon \right) $-almost paracontact structure, where the structure
vector field $\xi $ is spacelike or timelike according as $\varepsilon =1$
or $\varepsilon =-1$ \cite{Yuk-KTK-12}.

Motivated by these studies, in the present paper we investigate $\eta $-Ricci
solitons in $\left( \varepsilon \right) $-almost paracontact metric
manifolds. The paper is organized as follows. Section 2 is devoted to basic
concepts on $\left( \varepsilon \right) $-almost paracontact metric
manifolds. In Section 3, we study $\eta $-Ricci solitons in the case
when its potential vector field is exactly the characteristic vector field $%
\xi $ of the Einstein-like $( \varepsilon )$-almost paracontact metric manifold and
when the potential vector field is torse-forming in an $\eta $-Einstein $%
(\varepsilon )$-almost paracontact metric manifold. In Section 4, we prove
that an $\left( \varepsilon \right) $-para Sasakian manifold admitting $\eta
$-Ricci soliton with a potential vector field pointwise collinear to $%
\xi $ is an Einstein-like manifold. In Section 5 we give some
characterizations for $\eta $-Ricci solitons on $(\varepsilon )$%
-almost paracontact metric manifolds concerning parallel symmetric
(0,2)-tensor fields.

\section{Preliminaries}

Let $M$ be an $n$-dimensional manifold equipped with an \textit{almost paracontact structure} $(\varphi ,\xi ,\eta
)$ \cite%
{Sato-76} consisting of a tensor field $\varphi $ of type $(1,1)$, a vector field $%
\xi $ and a $1$-form $\eta $ satisfying
\begin{equation}
\varphi ^{2}=I-\eta \otimes \xi ,  \label{eq-phi-eta-xi}
\end{equation}%
\begin{equation}
\eta (\xi )=1,  \label{eq-eta-xi}
\end{equation}%
\begin{equation}
\varphi \xi =0,  \label{eq-phi-xi}
\end{equation}%
\begin{equation}
\eta \circ \varphi =0.  \label{eq-eta-phi}
\end{equation}%
It is easy to verify that (\ref{eq-phi-eta-xi}) and one of (\ref{eq-eta-xi}%
), (\ref{eq-phi-xi}) and (\ref{eq-eta-phi}) imply the other two equations.
If $g$ is a pseudo-Riemannian metric such that
\begin{equation}
g\left( \varphi X,\varphi Y\right) =g\left( X,Y\right) -\varepsilon \eta
(X)\eta \left( Y\right), \qquad X,Y\in \Gamma (TM),  \label{eq-metric-1}
\end{equation}%
where $\varepsilon =\pm 1$, then $M$ is called $\left( \varepsilon
\right) ${\it -almost paracontact metric manifold} equipped with an \textit{$\left(
\varepsilon \right) ${\em -}almost paracontact metric structure} $(\varphi
,\xi ,\eta ,g,\varepsilon )$ \cite{Tri-KYK-10}. In particular, if ${\rm index%
}(g)=1$, that is when $g$ is a Lorentzian metric, then the $(\varepsilon )$%
-almost paracontact metric manifold is called {\it Lorentzian almost
paracontact manifold}. From (\ref{eq-metric-1}) we have
\begin{equation}
g\left( X,\xi \right) =\varepsilon \eta (X)  \label{eq-metric-3}
\end{equation}%
\begin{equation}
g\left( X,\varphi Y\right) =g\left( \varphi X,Y\right),  \label{eq-metric-2}
\end{equation}%
for all $X,Y\in \Gamma (TM)$. From (\ref{eq-metric-3}) it follows that
\begin{equation}
g\left( \xi ,\xi \right) =\varepsilon,  \label{eq-g(xi,xi)}
\end{equation}%
that is, the structure vector field $\xi $ is never lightlike.

\bigskip

Let $(M,\varphi ,\xi ,\eta ,g,\varepsilon )$ be an $(\varepsilon )$-almost
paracontact metric manifold (resp. a Lorentzian almost paracontact
manifold). If $\varepsilon =1$, then $M$ is said to be a spacelike $%
(\varepsilon )$-almost paracontact metric manifold (resp. a spacelike
Lorentzian almost paracontact manifold). Similarly, if $\varepsilon =-\,1$,
then $M$ is said to be a timelike $(\varepsilon )$-almost paracontact
metric manifold (resp. a timelike Lorentzian almost paracontact manifold)
\cite{Tri-KYK-10}.

\bigskip

An $\left( \varepsilon \right) $-almost paracontact metric structure $(\varphi ,\xi ,\eta ,g,\varepsilon )$ is
called $\left( \varepsilon \right) ${\it -para Sasakian structure} if
\begin{equation}
(\nabla _{X}\varphi )Y=-\,g(\varphi X,\varphi Y)\xi -\varepsilon \eta \left(
Y\right) \varphi ^{2}X,\qquad X,Y\in \Gamma (T\!M),  \label{para2}
\end{equation}%
where $\nabla $ is the Levi-Civita connection with respect to $g$. A
manifold endowed with an $\left( \varepsilon \right) $-para Sasakian
structure is called $\left( \varepsilon \right) ${\it -para Sasakian
manifold} \cite{Tri-KYK-10}. In an $\left( \varepsilon \right) ${\em -}para
Sasakian manifold, we have
\begin{equation}
\nabla \xi =\varepsilon \varphi  \label{para3}
\end{equation}%
and the Riemann
curvature tensor $R$ and the Ricci tensor $S$ satisfy the following
equations \cite{Tri-KYK-10}:%
\begin{equation}
R\left( X,Y\right) \xi =\eta \left( X\right) Y-\eta \left( Y\right) X,
\label{eq-eps-PS-R(X,Y)xi}
\end{equation}%
\begin{equation}
R\left( \xi ,X\right) Y=-\,\varepsilon g\left( X,Y\right) \xi +\eta \left(
Y\right) X,  \label{eq-eps-PS-R(xi,X)Y}
\end{equation}%
\begin{equation}
\eta \left( R\left( X,Y\right) Z\right) =-\,\varepsilon \eta \left( X\right)
g\left( Y,Z\right) +\varepsilon \eta \left( Y\right) g\left( X,Z\right),
\label{eq-eps-PS-eta(R(X,Y),Z)}
\end{equation}%
\begin{equation}
S(X,\xi )=-(n-1)\eta (X),  \label{eq-eps-PS-S(X,xi)}
\end{equation}
for all $X,Y,Z\in \Gamma (TM)$.

\begin{example}
\cite{Tri-KYK-10} Let ${\Bbb R}^{5}$\ be the $5$-dimensional real number
space with a coordinate system $\left( x,y,z,t,s\right) $. Defining
\[
\eta =ds-ydx-tdz\ ,\qquad \xi =\frac{\partial }{\partial s}\,
\]%
\[
\varphi \left( \frac{\partial }{\partial x}\right) =-\,\frac{\partial }{%
\partial x}-y\frac{\partial }{\partial s}\ ,\qquad \varphi \left( \frac{%
\partial }{\partial y}\right) =-\,\frac{\partial }{\partial y}\,
\]%
\[
\varphi \left( \frac{\partial }{\partial z}\right) =-\,\frac{\partial }{%
\partial z}-t\frac{\partial }{\partial s}\ ,\qquad \varphi \left( \frac{%
\partial }{\partial t}\right) =-\,\frac{\partial }{\partial t}\ ,\qquad
\varphi \left( \frac{\partial }{\partial s}\right) =0\,
\]%
\[
g_{1}=\left( dx\right) ^{2}+\left( dy\right) ^{2}+\left( dz\right)
^{2}+\left( dt\right) ^{2}-\eta \otimes \eta \,
\]%
\begin{eqnarray*}
g_{2} &=&-\,\left( dx\right) ^{2}-\left( dy\right) ^{2}+\left( dz\right)
^{2}+\left( dt\right) ^{2}+\left( ds\right) ^{2} \\
&&-\,t\left( dz\otimes ds+ds\otimes dz\right) -y\left( dx\otimes
ds+ds\otimes dx\right),
\end{eqnarray*}%
then $(\varphi ,\xi ,\eta ,g_{1})$ is a timelike Lorentzian almost
paracontact structure in ${\Bbb R}^{5}$, while $(\varphi ,\xi ,\eta
,g_{2})$\ is a spacelike $\left( \varepsilon \right) $-almost paracontact
structure. Note that {\rm index}$\left( g_{2}\right) =3$.
\end{example}

\section{$\protect\eta $-Ricci solitons on Einstein-like $(\protect%
\varepsilon )$-almost paracontact metric manifolds}

We introduce the following definition analogous to Einstein-like para
Sasakian manifolds \cite{Sharma-82}.

\begin{definition}
An $\left( \varepsilon \right) $-almost paracontact metric manifold $\left(
M,\varphi ,\xi ,\eta ,g,\varepsilon \right) $ is said to be {\it %
Einstein-like} if its Ricci tensor $S$ satisfies
\begin{equation}
S\left( X,Y\right) =a\,g\left( X,Y\right) +b\,g\left( \varphi X,Y\right)
+c\,\eta \left( X\right) \eta \left( Y\right),\qquad X,Y\in \Gamma (TM)  \label{1}
\end{equation}%
for some real constants $a$, $b$ and $c$.
\end{definition}

We deduce the following properties:

\begin{proposition}
In an Einstein-like $\left( \varepsilon \right) $-almost paracontact metric
manifold \linebreak $\left( M,\varphi ,\xi ,\eta ,g,\varepsilon ,a,b,c\right) $ we
have
\begin{equation}
S\left( \varphi X,Y\right) =S(X,\varphi Y),  \label{2}
\end{equation}%
\begin{equation}
S\left( \varphi X,\varphi Y\right) =S(X,Y)-(\varepsilon a+c)\eta \left(
X\right) \eta \left( Y\right),  \label{3}
\end{equation}%
\begin{equation}
S\left( X,\xi \right) =(\varepsilon a+c)\eta \left( X\right),  \label{4}
\end{equation}%
\begin{equation}
S\left( \xi ,\xi \right) =\varepsilon a+c,  \label{5}
\end{equation}%
\begin{equation}
\left( \nabla _{X}\,S\right) (Y,Z)=bg((\nabla _{X}\varphi )Y,Z)+\varepsilon
c\left\{ \eta (Y)g(\nabla _{X}\xi ,Z)+\eta (Z)g(\nabla _{X}\xi ,Y)\right\},
\label{5a}
\end{equation}%
\begin{equation}
\left( \nabla _{X}\,Q\right) Y=b(\nabla _{X}\varphi )Y+\varepsilon c\left\{
\eta (Y)\nabla _{X}\xi +\varepsilon g(\nabla _{X}\xi ,Y)\xi \right\},
\label{5b}
\end{equation}%
where $Q$ is the Ricci operator defined by $g(QX,Y)=S(X,Y),$ $X,Y\in \Gamma
(TM)$. Moreover, if the manifold is $\left( \varepsilon \right) $-para
Sasakian, then
\begin{equation}
\varepsilon a+c=1-n,  \label{6}
\end{equation}%
\begin{equation}
r=na+b\,{\rm trace}(\varphi )+\varepsilon c,  \label{7}
\end{equation}%
where $r$\ is the scalar curvature.
\end{proposition}

Remark that the Ricci operator $Q$ of an Einstein-like $\left(
\varepsilon \right) $-almost paracontact metric manifold is of the form%
\[
Q=aI+b\varphi +\varepsilon c\eta \otimes \xi
\]%
and the structure vector field $\xi $ is an eigenvector of $Q$ with the
corresponding eigenvalue $a+\varepsilon c.$

\bigskip

Let $\left( M,\varphi ,\xi ,\eta ,g,\varepsilon ,a,b,c\right) $ be an
Einstein-like $\left( \varepsilon \right) $-almost paracontact metric
manifold admitting an\textit{ $\eta $-Ricci soliton}, that is, a tuple $\left( g,\xi
,\lambda ,\mu \right) $ satisfying%
\begin{equation}
\frac{1}{2}\pounds _{\xi }g+S+\lambda g+\mu \eta \otimes \eta =0,  \label{10}
\end{equation}%
with $\lambda $ and $\mu $ real constants. Replacing (\ref{1}) in the last
equation we get%
\begin{equation}
g(\nabla _{X}\xi ,Y)+g(\nabla _{Y}\xi ,X)+2\left\{ \left( a+\lambda \right)
g(X,Y)+bg(\varphi X,Y)+(c+\mu )\eta (X)\eta (Y)\right\} =0,  \label{11}
\end{equation}%
for all $X,Y\in \Gamma (TM).$ If we take $X=Y=\xi $ in (\ref{11}) we have%
\begin{equation}
\varepsilon (a+\lambda )+c+\mu =0,  \label{12}
\end{equation}%
by virtue of (\ref{eq-eta-xi}) and (\ref{eq-g(xi,xi)}).

Using (\ref{12}) and taking $Y=\xi $ in (\ref{11}), we
obtain
\begin{equation}
g(\nabla _{\xi }\xi ,X)=0,  \label{13}
\end{equation}%
which implies $\nabla _{\xi }\xi =0$. So we easily see that
\begin{equation}
\left( \nabla _{\xi }\varphi \right) \xi =0\text{ \ \ \ \ and \ \ \ \ }%
\nabla _{\xi }\eta =0.  \label{14}
\end{equation}

Also from (\ref{5a}), (\ref{5b}) and (\ref{14}) we get%
\begin{equation}
\left( \nabla _{\xi }\,S\right) (Y,Z)=bg((\nabla _{\xi }\varphi )Y,Z)
\label{15}
\end{equation}%
and%
\begin{equation}
\nabla _{\xi }\,Q=b\nabla _{\xi }\varphi.  \label{16}
\end{equation}

Hence we have the following result:

\begin{proposition}
Let $\left( M,\varphi ,\xi ,\eta ,g,\varepsilon ,a,b,c\right) $ be an
Einstein-like $\left( \varepsilon \right) $-almost paracontact metric
manifold admitting an $\eta $-Ricci soliton $\left( g,\xi ,\lambda ,\mu
\right) $. Then

i) $\varepsilon (a+\lambda )+c+\mu =0,$

ii) $\xi $ is a geodesic vector field,

iii) $\left( \nabla _{\xi }\varphi \right) \xi =0$ \ and $\ \nabla
_{\xi }\eta =0,$

iv) $\left( \nabla _{\xi }\,S\right) (Y,Z)=bg((\nabla _{\xi }\varphi
)Y,Z)$ \ and \ $\nabla _{\xi }\,Q=b\nabla _{\xi }\varphi.$

Moreover, if the manifold is $\left( \varepsilon \right) $-para Sasakian, then
\[
\nabla _{\xi }\,S=0\text{\ \ and \ \ }\nabla _{\xi }Q=0.
\]%
\end{proposition}

\bigskip

A vector field $\xi $ is called \textit{torse-forming} if
\begin{equation}
\nabla _{X}\xi =fX+w(X)\xi,  \label{17}
\end{equation}%
is satisfied for some smooth function $f$ and a $1$-form $w$.

Taking the inner product with $\xi$ we
have
\[
0=g(\nabla _{X}\xi ,\xi )=\varepsilon \left( f\eta (X)+w(X)\right),
\]%
for all $X\in \Gamma (TM)$, which implies%
\begin{equation}
w=-f\eta.  \label{19}
\end{equation}%
It follows
\begin{equation}
\nabla _{X}\xi =f\left( X-\eta (X)\xi \right) =f\varphi ^{2}X.  \label{20}
\end{equation}%

\bigskip

Now assume that $\left( M,\varphi ,\xi ,\eta ,g,\varepsilon ,a,b,c\right) $
is an Einstein-like $\left( \varepsilon \right) $-almost paracontact metric
manifold admitting an $\eta $-Ricci soliton $\left( g,\xi ,\lambda ,\mu
\right) $ and that the potential vector field $\xi $ is torse-forming.
Replacing (\ref{20}) in (\ref{11}) we obtain, for all $X,Y\in \Gamma (TM)$
\[
0 =(f+a+\lambda )\left\{ g(X,Y)-\varepsilon \eta (X)\eta (Y)\right\}
+bg(\varphi X,Y)
\]%
and
\[
0 =g((f+a+\lambda )\varphi X+bX,\varphi Y),
\]%
which implies%
\begin{equation}
0=(f+a+\lambda )\varphi ^{2}X+b\varphi X, \label{26}
\end{equation}%
that is
\begin{equation}
b\varphi X=-(f+a+\lambda )X+(f+a+\lambda )\eta (X)\xi.  \label{21a}
\end{equation}

So we have:

\begin{theorem}
Let $\left( M,\varphi ,\xi ,\eta ,g,\varepsilon ,a,b,c\right) $ be an
Einstein-like $\left( \varepsilon \right) $-almost paracontact metric
manifold admitting an $\eta $-Ricci soliton $\left( g,\xi ,\lambda ,\mu
\right) $ with torse-forming potential vector field. Then $M$ is an $\eta $%
-Einstein manifold.
\end{theorem}

\bigskip

In the remaining part of this section, we shall consider $M$ an \textit{$\eta $-Einstein manifold} (that is an Einstein-like $\left( \varepsilon \right) $-almost paracontact metric
manifold with $b=0$), admitting an $\eta $-Ricci
soliton $\left( g,\xi ,\lambda ,\mu \right) $ with torse-forming potential
vector field $\xi $. Using (\ref{26}) we have
\[
f=-a-\lambda.
\]%
So we can write
\[
\nabla _{X}\xi =-\left( a+\lambda \right) \left( X-\eta (X)\xi \right)
=-\left( a+\lambda \right) \varphi ^{2}X.
\]%
By using (\ref{20}) we obtain%
\begin{equation}
R(X,Y)\xi =(a+\lambda )^{2}\{\eta (X)Y-\eta (Y)X\}  \label{24}
\end{equation}%
and
\begin{equation}
S(X,\xi )=(a+\lambda )^{2}(1-n)\eta (X),  \label{25}
\end{equation}%
for all $X,Y\in \Gamma (TM).$ From (\ref{4}) and (\ref{25}) we get
\[
(\varepsilon a+c)\eta \left( X\right) =(a+\lambda )^{2}(1-n)\eta (X),
\]%
which implies
\begin{equation}
c=-\varepsilon a+(a+\lambda )^{2}(1-n).  \label{25a}
\end{equation}%
Also by using (\ref{12}) in the last equation we get%
\begin{equation}
\mu =-\varepsilon \left( \lambda +\varepsilon (a+\lambda )^{2}(1-n)\right).
\label{27}
\end{equation}%

So we have:

\begin{theorem}
Let $\left( M,\varphi ,\xi ,\eta ,g,\varepsilon ,a,c\right) $ be an $\eta $%
-Einstein $\left( \varepsilon \right) $-almost paracontact metric manifold
admitting an $\eta $-Ricci soliton $\left( g,\xi ,\lambda ,\mu \right) $
with torse-forming potential vector field. Then $f$ is a constant
function and
\[
c=-\varepsilon a+(a+\lambda )^{2}(1-n),
\]%
\[
\mu =-\varepsilon \left( \lambda +\varepsilon (a+\lambda )^{2}(1-n)\right).
\]
\end{theorem}

\begin{proposition}
In an $\eta $-Einstein $\left( \varepsilon \right) $-almost paracontact
metric manifold admitting an $\eta $-Ricci soliton with torse-forming
potential vector field, we have
\begin{equation}
(\nabla _{X}S)(Y,Z)=-c\varepsilon \left( a+\lambda \right) \left\{ \eta
(Y)g(X,Z)+\eta (Z)g(X,Y)-2\varepsilon \eta (X)\eta (Y)\eta (Z)\right\}
\end{equation}
and
\begin{equation}
\left( \nabla _{X}Q\right) Y=-c\left( a+\lambda \right) \left\{ \eta
(Y)X+\varepsilon g(X,Y)\xi -2\eta (X)\eta (Y)\xi \right\}. \label{38}
\end{equation}
\end{proposition}

\begin{theorem}
Let $\left( M,\varphi ,\xi ,\eta ,g,\varepsilon ,a,c\right) $ be an $\eta $%
-Einstein $\left( \varepsilon \right) $-almost paracontact metric manifold
admitting an $\eta $-Ricci soliton $\left( g,\xi ,\lambda ,\mu \right) $
with torse-forming potential vector field. If $f\neq 0$ and the Ricci operator $Q$ is Codazzi, then $M$ is an Einstein
manifold and $\xi $ is a Killing vector field.
\end{theorem}

\begin{proof}
From the condition
\begin{equation}
\left( \nabla _{X}Q\right) Y=\left( \nabla _{Y}Q\right) X,  \label{60}
\end{equation}%
for all $X,Y\in \Gamma (TM)$, using (\ref{38}) we get
\[
c(a+\lambda)\{\eta(X)Y-\eta(Y)X\}=0,
\]%
for all $X,Y\in \Gamma (TM)$. Since $a+\lambda=-f\neq 0$ we obtain $c=0$ and hence, $M$ is Einstein manifold.

Moreover, writing (\ref{5b}) for $Y=\xi $ we obtain
\[
\nabla _{X}\xi =0.
\]%
Therefore, $L_{\xi}g=0$, hence $\xi$ is Killing vector field.
\end{proof}

\bigskip

Let us remark the following particular cases:

Case I: $f=-1.$ In this case, $\xi $ is an irrotational vector field and we
have
\[
\nabla \xi =-I+\eta \otimes \xi,
\]%
\[
\lambda =1-a,
\]%
\[
\mu =-\varepsilon \left( 1+\varepsilon c\right),
\]%
\[
R(X,Y)\xi =\eta (X)Y-\eta (Y)X.
\]%
Since $\eta \neq 0$, from (\ref{19}) it is easy to see that $\xi $ can not
be a concurrent vector field.

Case II: $f=0.$ In this case, $\xi $ is a recurrent vector field and we
have
\[
\nabla \xi =0,
\]%
\[
\lambda =-a,
\]%
\[
\mu =\varepsilon a=-c,
\]%
\[
R(X,Y)\xi =0.
\]%
Furthermore, $S$ and $Q$ are $\nabla $-parallel.

\begin{theorem}
On an $n$-dimensional $(n>1)$ non-Ricci flat $\eta $-Einstein $(\varepsilon
) $-almost paracontact manifold with $a=0$ admitting a torse-forming Ricci soliton $(g,\xi,\lambda)$ we have
\[
\lambda =\frac{\varepsilon }{n-1}\quad \text{and \ \ }c=-\frac{1}{n-1}.
\]%
\end{theorem}

\begin{corollary}
A torse-forming Ricci soliton on an $n$-dimensional $(n>1)$ non-Ricci flat $%
\eta $-Einstein spacelike (resp. timelike) $(\varepsilon )$-almost
paracontact manifold with $a=0$ is expanding (resp. shrinking).
\end{corollary}

\section{$\protect\eta $-Ricci solitons on $\left( \protect\varepsilon %
\right) $-para Sasakian manifolds}

Let $\left( M,\varphi ,\xi ,\eta ,g,\varepsilon \right) $ be an $\left(
\varepsilon \right) $-para Sasakian manifold admitting an $\eta $-Ricci
soliton $\left( g,V,\lambda ,\mu \right) $ and assume that the potential
vector field $V$ is pointwise collinear with the structure vector field $\xi
$, that is, $V=k\xi $, for $k$ a smooth function on $M.$ Then from (\ref{int-2}%
) and (\ref{eq-metric-3}) we have
\begin{equation}
\varepsilon (Xk)\eta (Y)+\varepsilon (Yk)\eta (X)+2\varepsilon kg(\varphi
X,Y)+2S(X,Y)+2\lambda g(X,Y)+2\mu \eta (X)\eta (Y)=0,  \label{31}
\end{equation}%
for all $X,Y\in \Gamma (TM).$ Taking $Y=\xi $ in (\ref{31}) and using (\ref%
{eq-eps-PS-S(X,xi)}) we get%
\begin{equation}
\varepsilon \left( Xk\right) +\left\{ \varepsilon \left( \xi k\right)
-2(n-1)+2\varepsilon \lambda +2\mu \right\} \eta (X)=0.  \label{32}
\end{equation}%
If we replace $X$ by $\xi $ in the last equation we obtain%
\begin{equation}
\xi k=\varepsilon (n-1)-\lambda -\varepsilon \mu.  \label{33}
\end{equation}%
Using (\ref{33}) in (\ref{32}) gives%
\[
Xk=\left( \varepsilon (n-1)-\lambda -\varepsilon \mu \right) \eta (X).
\]%
We conclude that $k$ is constant if $\varepsilon (n-1)=\lambda +\varepsilon \mu $
and in this case, from (\ref{31}) we get%
\[
S(X,Y)=-\lambda g(X,Y)-\varepsilon kg(\varphi X,Y)-\mu \eta (X)\eta (Y).
\]

So we have:

\begin{theorem}\label{t1}
Let $\left( M,\varphi ,\xi ,\eta ,g,\varepsilon \right) $ be an $%
(\varepsilon )$-para Sasakian manifold. If $M$ admits an $\eta $-Ricci
soliton $\left( g,V,\lambda ,\mu \right) $ and $V$ is pointwise collinear
with the structure vector field $\xi ,$ then $V$ is a constant multiple of $%
\xi $ provided $\varepsilon (n-1)=\lambda +\varepsilon \mu $ and $M$ is an
Einstein-like manifold.
\end{theorem}

\begin{remark}
Under the hypotheses of Theorem \ref{t1}, if $R(\xi,\cdot)\cdot S=0$, then $V$ is a constant multiple of $\xi$. Indeed, the condition on $S$ is
\[
S(R(\xi,X)Y,Z)+S(Y, R(\xi,X)Z)=0,
\]%
for all $X,Y,Z\in \Gamma (TM).$ Using (\ref{int-2}), (\ref{para3}) and (\ref{eq-eps-PS-R(xi,X)Y}) we obtain
\[
(\varepsilon(n-1)-\lambda)\{\eta(Y)g(X,Z)+\eta(Z)g(X,Y)\}-\varepsilon \{\eta(Y)g(\varphi X,Z)+\eta(Z)g(\varphi X,Y)\}-$$$$-2\mu \eta(X)\eta(Y)\eta(Z)=0,
\]%
for all $X,Y,Z\in \Gamma (TM)$ and taking $X=Y=Z=\xi$ we get
\[
\varepsilon(n-1)-\lambda-\varepsilon \mu=0.
\]%
\end{remark}

Assuming $V=\xi$ we get:

\begin{theorem}
On an $n$-dimensional $(n>1)$ $%
(\varepsilon )$-para Sasakian manifold admitting a Ricci soliton $(g,\xi,\lambda)$ we have
\[
\lambda =\frac{\varepsilon }{n-1}.
\]%
\end{theorem}

\begin{corollary}
A Ricci soliton on an $n$-dimensional $(n>1)$ $%
(\varepsilon )$-para Sasakian spacelike (resp. timelike) manifold is expanding (resp. shrinking).
\end{corollary}

\begin{remark}If we assume that $M$ is an Einstein-like $(\varepsilon )$-para Sasakian
manifold and $V=\xi $, we have
\[
\frac{1}{2}\left( \pounds _{\xi }\,g\right) \left( X,Y\right)
+S(X,Y)+\lambda g(X,Y)+\mu \eta (X)\eta (Y)=$$$$=(\varepsilon +b)g(\varphi
X,Y)+(a+\lambda )g(X,Y)+(c+\mu )\eta (X)\eta (Y),
\]%
which implies that if
\[
\varepsilon +b=0, \ \ a+\lambda=0, \ \ c+\mu=0,
\]%
then $\left( g,\xi ,-a,-c\right)$ is an $\eta $-Ricci soliton on $M$.
\end{remark}

\bigskip

We end these considerations by giving two examples of $\eta$-Ricci solitons on the $(\varepsilon )$-para Sasakian manifold considered in Example 5.2. from \cite{Tri-KYK-10}.

\begin{example}
Let $M={\Bbb R}^3$ and $(x,y,z)$ be the standard coordinates in ${\Bbb R}^3$. Set
$$\varphi:=\frac{\partial}{\partial x}\otimes dx-\frac{\partial}{\partial y}\otimes dy, \ \ \xi:=\frac{\partial}{\partial z}, \ \ \eta:=dz,$$
$$g:=e^{2z}dx\otimes dx+e^{-2z}dy\otimes dy+dz\otimes dz$$
and consider the orthonormal system of vector fields
$$E_1:=e^{-z}\frac{\partial}{\partial x}, \ \ E_2:=e^z\frac{\partial}{\partial y}, \ \ E_3:=\frac{\partial}{\partial z}.$$
Follows
$$\nabla_{E_1}E_1=-E_3, \ \ \nabla_{E_1}E_2=0, \ \ \nabla_{E_1}E_3=E_1, \ \ \nabla_{E_2}E_1=0, \ \ \nabla_{E_2}E_2=E_3,$$$$\nabla_{E_2}E_3=-E_2, \ \ \nabla_{E_3}E_1=0, \ \ \nabla_{E_3}E_2=0, \ \ \nabla_{E_3}E_3=0.$$
Then the Riemann and the Ricci curvature tensor fields are given by:
$$R(E_1,E_2)E_2=E_1, \ \ R(E_1,E_3)E_3=-E_1, \ \ R(E_2,E_1)E_1=E_2,$$ $$R(E_2,E_3)E_3=-E_2, \ \ R(E_3,E_1)E_1=-E_3, \ \ R(E_3,E_2)E_2=-E_3,$$
$$S(E_1,E_1)=0, \ \ S(E_2,E_2)=0, \ \ S(E_3,E_3)=-2.$$
In this case, for $\lambda=0$ and $\mu=2$, the data $(g,\xi,\lambda,\mu)$ is an $\eta$-Ricci soliton on the para Sasakian manifold $({\Bbb R}^3, \varphi , \xi , \eta , g)$.
\end{example}

\bigskip

\begin{example}
Let $M={\Bbb R}^3$ and $(x,y,z)$ be the standard coordinates in ${\Bbb R}^3$. Set
$$\varphi:=\frac{\partial}{\partial x}\otimes dx-\frac{\partial}{\partial y}\otimes dy, \ \ \xi:=\frac{\partial}{\partial z}, \ \ \eta:=dz,$$
$$g:=e^{-2z}dx\otimes dx+e^{2z}dy\otimes dy-dz\otimes dz$$
and consider the orthonormal system of vector fields
$$E_1:=e^{z}\frac{\partial}{\partial x}, \ \ E_2:=e^{-z}\frac{\partial}{\partial y}, \ \ E_3:=\frac{\partial}{\partial z}.$$
Follows
$$\nabla_{E_1}E_1=-E_3, \ \ \nabla_{E_1}E_2=0, \ \ \nabla_{E_1}E_3=-E_1, \ \ \nabla_{E_2}E_1=0, \ \ \nabla_{E_2}E_2=E_3,$$$$\nabla_{E_2}E_3=E_2, \ \ \nabla_{E_3}E_1=0, \ \ \nabla_{E_3}E_2=0, \ \ \nabla_{E_3}E_3=0.$$
Then the Riemann and the Ricci curvature tensor fields are given by:
$$R(E_1,E_2)E_2=-E_1, \ \ R(E_1,E_3)E_3=-E_1, \ \ R(E_2,E_1)E_1=-E_2,$$ $$R(E_2,E_3)E_3=-E_2, \ \ R(E_3,E_1)E_1=E_3, \ \ R(E_3,E_2)E_2=E_3,$$
$$S(E_1,E_1)=-2, \ \ S(E_2,E_2)=-2, \ \ S(E_3,E_3)=-2.$$
In this case, for $\lambda=2$ and $\mu=4$, the data $(g,\xi,\lambda,\mu)$ is an $\eta$-Ricci soliton on the Lorentzian para Sasakian manifold $({\Bbb R}^3, \varphi , \xi , \eta , g)$.
\end{example}

\section{Parallel symmetric $(0,2)$-tensor fields on $\left(
\protect\varepsilon \right) $-almost paracontact metric manifolds}

Let $\alpha $ be a $(0,2)$-tensor field which is assumed to be parallel with
respect to Levi-Civita connection $\nabla $, that is $\nabla \alpha =0$.
Applying the Ricci identity
\[
\nabla ^{2}\alpha (X,Y;Z,W)-\nabla ^{2}\alpha (X,Y;W,Z)=0,
\]%
we have \cite{Einshart26}:
\begin{equation}
\alpha (R(X,Y)Z,W)+\alpha (R(X,Y)W,Z)=0.  \label{50}
\end{equation}%
Taking $Z=W=\xi $ and using the symmetry property of $\alpha ,$ we write%
\begin{equation}
\alpha (R(X,Y)\xi ,\xi )=0.  \label{51}
\end{equation}%

Assume that $\left( M,\varphi ,\xi ,\eta ,g,\varepsilon \right) $ is an $%
\left( \varepsilon \right) $-almost paracontact metric manifold with
torse-forming characteristic vector field. Then from (\ref{20}) we have
\begin{equation}
R(X,Y)\xi =f^{2}\{\eta (X)Y-\eta (Y)X\}+X(f)\varphi ^{2}Y-Y(f)\varphi ^{2}X.
\label{52}
\end{equation}%
Replacing (\ref{52}) in (\ref{51}) we get
\begin{equation}
f^{2}\{ \eta (X)\alpha \left( Y,\xi \right) -\eta (Y)\alpha \left( X,\xi
\right) \} +X(f)\alpha \left( \varphi ^{2}Y,\xi \right) -Y(f)\alpha
\left( \varphi ^{2}X,\xi \right) =0.  \label{53}
\end{equation}%
If we take $X=\xi $ in (\ref{53}) we obtain
\begin{equation}
\left( f^{2}+\xi (f)\right) \left\{ \alpha \left( Y,\xi \right) -\eta
(Y)\alpha \left( \xi ,\xi \right) \right\} =0.  \label{54}
\end{equation}%
Let $f^{2}+\xi (f)\neq 0$; then we have
\begin{equation}
\alpha \left( Y,\xi \right) =\eta (Y)\alpha \left( \xi ,\xi \right).
\label{55}
\end{equation}

\begin{definition}
An $\left( \varepsilon \right) $-almost paracontact metric manifold $\left(
M,\varphi ,\xi ,\eta ,g,\varepsilon \right) $ with torse-forming
characteristic vector field is called \textit{regular} if $f^{2}+\xi (f)\neq 0.$
\end{definition}

Since $\alpha $ is a parallel $(0,2)$-tensor field, then $\alpha \left( \xi
,\xi \right) $ is a constant. Taking the covariant derivative of (\ref{55}) with
respect to $X$ we derive
\begin{equation}
\alpha (\nabla _{X}Y,\xi )+f\left\{ \alpha (X,Y)-\eta (X)\eta (Y)\alpha
\left( \xi ,\xi \right) \right\} =X\left( \eta (Y)\right) \alpha \left( \xi
,\xi \right),  \label{56}
\end{equation}%
which implies%
\begin{eqnarray*}
f\left\{ \alpha (X,Y)-\eta (X)\eta (Y)\alpha \left( \xi ,\xi \right)
\right\}  &=&\varepsilon \left\{ X(g(Y,\xi )-g\left( \nabla _{X}Y,\xi
\right) \right\} \alpha \left( \xi ,\xi \right)  \\
&=&\varepsilon g\left( Y,\nabla _{X}\xi \right) \alpha \left( \xi ,\xi
\right)  \\
&=&\varepsilon f\left\{ g(X,Y)-\varepsilon \eta (X)\eta (Y)\right\} \alpha
\left( \xi ,\xi \right)
\end{eqnarray*}%
and we obtain
\begin{equation}
\alpha (X,Y)=\varepsilon g(X,Y)\alpha \left( \xi ,\xi \right).  \label{57}
\end{equation}

Therefore:

\begin{theorem}
A symmetric parallel second order covariant tensor in a regular $\left(
\varepsilon \right) $-almost paracontact metric manifold with torse-forming
characteristic vector field is a constant multiple of the metric tensor.
\end{theorem}

Applying this result to solitons, we deduce:

\begin{theorem}
Let $\left( M,\varphi ,\xi ,\eta ,g,\varepsilon \right) $ be a regular $\left(
\varepsilon \right) $-almost paracontact metric manifold with torse-forming
characteristic vector field. Then $\alpha:=\frac{1}{2}\left( \pounds _{\xi}\,g\right)+S+\mu \eta\otimes \eta$ (with $\mu$ a real constant) is parallel if and only if $(g,\xi,\lambda=-\varepsilon \alpha(\xi,\xi),\mu)$ is an $\eta$-Ricci soliton on $M$.
\end{theorem}

Assume that $\left( M,\varphi ,\xi ,\eta ,g,\varepsilon ,a,b,c\right) $ is
an Einstein-like $\left( \varepsilon \right) $-almost paracontact metric
manifold with torse-forming characteristic vector field. Then
\begin{equation}
\frac{1}{2}\left( \pounds _{\xi}\,g\right)(X,Y)+S(X,Y)+\mu \eta(X) \eta(Y)=$$$$=(f+a)g(X,Y)+bg(\varphi X,Y)+(c+\mu-\varepsilon f)\eta(X) \eta(Y). \label{34}
\end{equation}

\begin{theorem}\label{36}
Let $\left( M,\varphi ,\xi ,\eta ,g,\varepsilon ,a,b,c\right) $ be a regular Einstein-like $\left(
\varepsilon \right) $-almost paracontact metric manifold with torse-forming
characteristic vector field. Then $\alpha:=\frac{1}{2}\left( \pounds _{\xi}\,g\right)+S+\mu \eta\otimes \eta$ (with $\mu$ a real constant) is parallel if and only if $(g,\xi,\lambda=-(a+\varepsilon (c+\mu)),\mu)$ is an $\eta$-Ricci soliton on $M$.
\end{theorem}

\begin{proof}
From (\ref{34}) we get $\alpha(\xi,\xi)=\varepsilon (a+\varepsilon c)+\mu$, so $\lambda=-\varepsilon \alpha(\xi,\xi)=-(a+\varepsilon (c+\mu))$.
\end{proof}

\begin{theorem}
Let $\left( M,\varphi ,\xi ,\eta ,g,\varepsilon ,a,b,c\right) $ be a regular Einstein-like $\left(
\varepsilon \right) $-almost paracontact metric manifold with torse-forming
characteristic vector field. Then $\alpha:=\frac{1}{2}\left( \pounds _{\xi}\,g\right)+S$ is parallel if and only if $(g,\xi,\lambda=-(a+\varepsilon c))$ is an expanding (resp. shrinking) Ricci soliton on $M$ provided $a+\varepsilon c<0$ (resp. $a+\varepsilon c>0$).
\end{theorem}

\bigskip

Assume that $\left( M,\varphi ,\xi ,\eta ,g,\varepsilon\right) $ is
an $(\varepsilon)$-para Sasakian manifold. From (\ref{eq-eps-PS-R(X,Y)xi}) and (%
\ref{51}) we have%
\begin{equation}
\eta (X)\alpha (Y,\xi )-\eta (Y)\alpha (X,\xi) =0.  \label{58}
\end{equation}%
Taking $X=\xi $ and $Y=\varphi ^{2}Z$ in the last equation we obtain
\begin{equation}
0=\alpha (\varphi ^{2}Z,\xi )=\alpha
(Z,\xi )-\eta (Z)\alpha (\xi ,\xi ),  \label{39}
\end{equation}%
for all $Z\in \Gamma (TM).$\pagebreak

Since $\alpha $ is a parallel $(0,2)$-tensor field, then $\alpha \left( \xi
,\xi \right) $ is a constant. Taking the covariant derivative of (\ref{39}) with
respect to $X$ we derive
\begin{equation}
\alpha (\nabla _{X}Z,\xi )+\varepsilon \alpha(Z,\varphi X) =X(\eta (Z)) \alpha (\xi,\xi),  \label{56}
\end{equation}%
which implies%
\[
\alpha(\varphi X,Z)=\varepsilon g(\varphi X,Z)\alpha (\xi,\xi).
\]%
Taking $X=\varphi Y$ in the last equation we get
\begin{equation}
\alpha (Y,Z)=\varepsilon g(Y,Z)\alpha \left( \xi ,\xi \right).  \label{57}
\end{equation}

Therefore:

\begin{theorem}
On an $(\varepsilon)$-para Sasakian manifold, any
parallel symmetric (0,2)-tensor field is a constant multiple of the metric.
\end{theorem}

Applying this result to solitons, we deduce:

\begin{theorem}
Let $\left( M,\varphi ,\xi ,\eta ,g,\varepsilon\right) $ be an $(\varepsilon)$-para Sasakian manifold. Then $\alpha:=\frac{1}{2}\left( \pounds _{\xi}\,g\right)+S+\mu \eta\otimes \eta$ (with $\mu$ a real constant) is parallel if and only if $(g,\xi,\lambda=-\varepsilon \alpha(\xi,\xi),\mu)$ is an $\eta$-Ricci soliton on $M$.
\end{theorem}

\bigskip

Assume that $\left( M,\varphi ,\xi ,\eta ,g,\varepsilon, a,b,c\right) $ is
an Einstein-like $(\varepsilon)$-para Sasakian manifold. Then
\begin{equation}
\frac{1}{2}\left( \pounds _{\xi}\,g\right)(X,Y)+S(X,Y)+\mu \eta(X) \eta(Y)=$$$$=ag(X,Y)+(\varepsilon+b)g(\varphi X,Y)+(c+\mu)\eta(X) \eta(Y). \label{35}
\end{equation}

\begin{theorem}\label{37}
Let $\left( M,\varphi ,\xi ,\eta ,g,\varepsilon ,a,b,c\right) $ be an Einstein-like $\left(
\varepsilon \right) $-para Sasakian manifold. Then $\alpha:=\frac{1}{2}\left( \pounds _{\xi}\,g\right)+S+\mu \eta\otimes \eta$ (with $\mu$ a real constant) is parallel if and only if $(g,\xi,\lambda=-(a+\varepsilon (c+\mu)),\mu)$ is an $\eta$-Ricci soliton on $M$.
\end{theorem}

\begin{proof}
From (\ref{35}) we get $\alpha(\xi,\xi)=\varepsilon (a+\varepsilon c)+\mu$, so $\lambda=-\varepsilon \alpha(\xi,\xi)=-(a+\varepsilon (c+\mu))$.
\end{proof}

\begin{remark}
From Theorem \ref{36} and Theorem \ref{37} we notice that the parallelism of the symmetric (0,2)-tensor field $\alpha:=\frac{1}{2}\left( \pounds _{\xi}\,g\right)+S+\mu \eta\otimes \eta$ on an Einstein-like $\left(
\varepsilon \right) $-almost paracontact metric manifold $\left( M,\varphi ,\xi ,\eta ,g,\varepsilon ,a,b,c\right) $ which either is regular with torse-forming characteristic vector field or is $\left(
\varepsilon \right) $-para Sasakian, yields the same $\eta$-Ricci soliton (which depends only on the constants $a,c$ and $\mu$).
\end{remark}

\small{

\newpage

\noindent Adara Monica Blaga

\noindent Department of Mathematics, West University of Timi\c{s}oara

\noindent Bld. V. Parvan nr. 4, 300223, Timi\c{s}oara, Romania

\noindent Email: adarablaga@@yahoo.com

\bigskip

\noindent Selcen Y\"{u}ksel Perkta\c{s}

\noindent Department of Mathematics, Faculty of Arts and Sciences,

\noindent Ad\i yaman University

\noindent 02040, Ad\i yaman, Turkey

\noindent Email: sperktas@@adiyaman.edu.tr

\bigskip

\noindent Bilal Eftal Acet

\noindent Department of Mathematics, Faculty of Arts and Sciences,

\noindent Ad\i yaman University

\noindent 02040, Ad\i yaman, Turkey

\noindent Email: eacet@@adiyaman.edu.tr

\bigskip

\noindent Feyza Esra Erdo\u{g}an

\noindent Faculty of Education, Department of Elementary Education,

\noindent Ad\i yaman University

\noindent 02040, Ad\i yaman, Turkey

\noindent Email: ferdogan@@adiyaman.edu.tr

}


\end{document}